\documentclass[12pt,reqno]{amsart}
\usepackage[T1]{fontenc}
\usepackage{lmodern}
\usepackage{amsmath,amssymb,mathtools}
\usepackage[letterpaper,margin=1in]{geometry}
\usepackage{microtype}
\usepackage{enumitem}
\usepackage{needspace}
\usepackage{xcolor}
\usepackage[numbers,square,sort]{natbib}
\usepackage{hyperref}
\hypersetup{colorlinks=true,linkcolor=blue,citecolor=blue,urlcolor=black,pdfauthor={Lu Chen, Yan Jiang, Hongyu Liu, Longyue Tao},pdftitle={On an inverse boundary problem in elastodynamics},pdfsubject={Static isotropic elasticity; structured shear moduli; uniqueness}}
\numberwithin{equation}{section}
\newtheorem{theorem}{Theorem}[section]
\newtheorem{lemma}[theorem]{Lemma}
\newtheorem{proposition}[theorem]{Proposition}
\theoremstyle{definition}
\newtheorem{definition}{Definition}[section]
\theoremstyle{remark}
\newtheorem{remark}{Remark}[section]
\setlist[enumerate]{label=\textup{(\roman*)},leftmargin=2em,itemsep=0.35em,topsep=0.4em}
\allowdisplaybreaks[1]
\title[An inverse boundary problem in elastodynamics]{On an inverse boundary problem in elastodynamics}
\author{Lu Chen}
\address[Lu Chen]{Key Laboratory of Algebraic Lie Theory and Analysis, Ministry of Education, School of Mathematics and Statistics, Beijing Institute of Technology, Beijing 100081, PR China}
\email{chenlu5818804@163.com}
\author{Yan Jiang}
\address[Yan Jiang]{Department of Mathematics, City University of Hong Kong, Hong Kong SAR, China}
\email{yjian24@cityu.edu.hk}
\author{Hongyu Liu}
\address[Hongyu Liu]{Department of Mathematics, City University of Hong Kong, Hong Kong SAR, China}
\email{hongyu.liuip@gmail.com, hongyliu@cityu.edu.hk}
\author{Longyue Tao}
\address[Longyue Tao]{Department of Mathematics, City University of Hong Kong, Hong Kong SAR, China}
\email{sdyctly@163.com, longyue.tao@my.cityu.edu.hk}
\date{}

\begin{document}
\begin{abstract}
We study the simultaneous identification of the Lam\'e parameters $\lambda$ and $\mu$ in three-dimensional static isotropic elasticity.
For smooth coefficients, we establish uniqueness from the full Dirichlet-to-Neumann map when $\mu$ is axisymmetric, admits a multiplicative or additive separation on a product domain, or is quasianalytic in one fixed direction.
These results require neither smallness of $\nabla\mu$ nor real analyticity of the coefficients, and impose no corresponding structural condition on $\lambda$.
In the axisymmetric and separated cases, partial boundary data suffice to determine $\mu$.
\end{abstract}
\maketitle

\section{Introduction}\label{sec:introduction}
\subsection{Problem formulation}\label{subsec:problem-background}
Let $\Omega\subset\mathbb{R}^3$ be a known bounded connected Lipschitz domain.
For a displacement $u:\Omega\to\mathbb{R}^3$, define the strain tensor $\varepsilon(u)$ and the stress tensor $\sigma_{\lambda,\mu}(u)$ by
\begin{equation}\label{eq:strain-stress}
 \varepsilon(u)=\frac12\bigl(\nabla u+(\nabla u)^T\bigr),
 \qquad
 \sigma_{\lambda,\mu}(u)=\lambda(\operatorname{div} u)I+2\mu\varepsilon(u),
\end{equation}
Here $I$ is the $3\times3$ identity matrix, $\lambda$ and $\mu$ are the \emph{Lam\'e parameters}, and $\mu$ is the \emph{shear modulus}.
We assume that the coefficients are real valued and bounded, and satisfy the uniform strong convexity condition
\begin{equation}\label{eq:convexity}
 \mu\geq c_*,\qquad 3\lambda+2\mu\geq c_*
 \quad\hbox{almost everywhere in }\Omega
\end{equation}
for some $c_*>0$.
The \emph{elastic Dirichlet problem}
\begin{equation}\label{eq:lame-dirichlet}
 \mathcal{L}_{\lambda,\mu}u:=\operatorname{div}\sigma_{\lambda,\mu}(u)=0
 \quad\hbox{in }\Omega,
 \qquad u|_{\partial\Omega}=f
\end{equation}
has a unique weak solution $u_f\in H^1(\Omega)^3$ for every $f\in H^{1/2}(\partial\Omega)^3$ \cite{ciarlet1988mathematical}.
Its Dirichlet-to-Neumann map is
\begin{equation}\label{eq:elastic-dn-map}
 \Lambda_{\lambda,\mu}^{\Omega}:
 H^{1/2}(\partial\Omega)^3\longrightarrow H^{-1/2}(\partial\Omega)^3,
 \qquad
 \Lambda_{\lambda,\mu}^{\Omega}f=\sigma_{\lambda,\mu}(u_f)\nu|_{\partial\Omega},
\end{equation}
where $\nu$ is the outward unit normal and the traction is understood in the weak sense.
We consider the inverse problem of recovering both Lam\'e parameters from knowledge of $\Lambda_{\lambda,\mu}^{\Omega}$.
We are mainly concerned with the unique identifiability issue, namely, sufficient conditions under which
\begin{equation}\label{eq:identifiability}
 \Lambda_{\lambda_1,\mu_1}^{\Omega}
 =\Lambda_{\lambda_2,\mu_2}^{\Omega}
 \quad\Longrightarrow\quad
 (\lambda_1,\mu_1)=(\lambda_2,\mu_2)\quad\hbox{in }\Omega.
\end{equation}

We also consider \emph{partial boundary data}, for which only a nonempty relatively open subset $\Gamma$ of the smooth part of $\partial\Omega$ is accessible.
Displacements are prescribed on $\Gamma$ and set to zero on $\partial\Omega\setminus\Gamma$, and the resulting traction is measured only on $\Gamma$.
We denote the corresponding localized Dirichlet-to-Neumann map by $\Lambda_{\lambda,\mu}^{\Gamma}$; its precise definition is given in Section~\ref{subsec:weak-formulation}.

\subsection{Discussion}
Two established approaches to this problem are relevant to the present work.
Perturbative uniqueness is known when $\nabla\mu$ is sufficiently small in a sufficiently high $C^k$ norm, under the corresponding regularity and a priori bounds; see the results and corrected theorem in \cite{eskin2002inverse,eskin2003inverse,nakamura2003global}.
With partial boundary data, both parameters can be recovered when $\mu$ is constant and the boundary values of $\lambda$ are known on the inaccessible part \cite{imanuvilov2012uniqueness}.
Another approach combines boundary determination with real analyticity.
Under suitable smoothness hypotheses, boundary measurements determine the Taylor coefficients of the Lam\'e parameters \cite{lin2017boundary,nakamura1995inverse,tan2023determining}.
If both parameters are real analytic in a neighborhood of the closed connected domain, these boundary values and derivatives determine them throughout the domain by analytic continuation \cite{tan2023determining}.
These results motivate the search for uniqueness beyond nearly constant shear moduli and real-analytic Lam\'e parameters.

Our motivation is to understand how geometric structure, coefficient structure, and regularity complement one another in establishing unique identifiability.
In the anisotropic Calder\'on problem, conformal product geometries, together with suitable assumptions on the transverse geometry, provide a framework for studying smooth unknown conformal factors beyond the real-analytic setting \cite{dos2009limiting,dos2016calderon}.
Restrictions on the spatial dependence of the coefficients provide another source of structure, as illustrated by layered conductivities \cite{kohn1985determining} and warped-product metrics whose unknown warping functions depend only on an axial variable \cite{daude2020anisotropic}.
In these settings, the prescribed dependence allows uniqueness under weaker regularity assumptions.
Geometry and coefficient structure are also closely connected in inverse boundary and scattering problems on St\"ackel-type manifolds, where the coordinate dependence of the metric gives rise to separability of the governing equation \cite{daude2021anisotropic,gobin2018inverse}.
These examples motivate us to examine how the geometry of the domain and the spatial dependence of the shear modulus jointly support unique identifiability in isotropic elasticity.

The analytic continuation approach recalled above for elasticity also appears in inverse problems for real-analytic conductivities and Riemannian metrics \cite{lee1989determining,lassas2001determining}.
Quasianalyticity retains the underlying continuation principle beyond real analyticity \cite{thilliez2008quasianalytic,boman2010unique}.
Recent work has explored this principle through quasianalytic assumptions on interior metrics \cite{chen2026quasianalyticity} and spectral conditions supporting the propagation of partial boundary data \cite{daude2026local}.
These perspectives motivate us to investigate directional quasianalyticity as a condition imposed only on the shear modulus.

Our main results are as follows.
\begin{enumerate}
\item \textbf{Axisymmetric shear moduli.}
We establish unique identifiability of axisymmetric shear moduli on domains of revolution from full torsional boundary data in the bounded measurable setting, and from partial torsional data in the smooth setting.
\item \textbf{Separation of variables.}
We prove that partial boundary data uniquely determine a smooth shear modulus admitting either a multiplicative or an additive separation on a product domain, even when both factors or both summands are unknown.
\item \textbf{Directional quasianalyticity.}
We establish unique identifiability of smooth shear moduli that are quasianalytic in one fixed direction from full boundary data, without requiring the domain to be convex.
\end{enumerate}
In each of the smooth settings above, the full elastic Dirichlet-to-Neumann map uniquely determines both Lam\'e parameters, with the structural or quasianalytic condition imposed only on the shear modulus.

\subsection{Main results}\label{subsec:main-results}
We state our main results for axisymmetric, separated, and directionally quasianalytic shear moduli.

\subsubsection{Axisymmetric shear moduli}\label{subsec:axisymmetric-results}
Let $D\Subset\{(r,z)\in\mathbb{R}^2:r>0\}$ be a bounded connected domain with $C^\infty$ boundary, and define its solid of revolution by
\[
 \Omega_D=\{(r\cos\theta,r\sin\theta,z):(r,z)\in D,\ 0\leq\theta<2\pi\}.
\]
The assumption on $D$ means that $0<r_-\leq r\leq r_+<\infty$ on $\overline{D}$.
For an axisymmetric shear modulus, we write $\mu(r,\theta,z)=m(r,z)$, where $m$ is defined on $D$.
Write $e_\theta=(-\sin\theta,\cos\theta,0)$.
For $f\in H^{1/2}(\partial D)$, the torsional boundary lift $\mathcal{T}:H^{1/2}(\partial D)\to H^{1/2}(\partial\Omega_D)^3$ is given by
\begin{equation}\label{eq:torsional-boundary-displacement}
 \mathcal{T}f=r f(r,z)e_\theta.
\end{equation}
The boundedness of $\mathcal{T}$ is proved in Proposition~\ref{prop:torsion}.
If $\widetilde\Gamma\subset\partial D$ is a nonempty relatively open arc, let $\Gamma\subset\partial\Omega_D$ be the surface obtained by rotating that arc.
Define the torsional measurement form
\[
 \mathcal{M}_{\lambda,\mu}(f,g)
 =\left\langle \Lambda_{\lambda,\mu}^{\Omega_D}\mathcal{T}f,\mathcal{T}g\right\rangle,
 \qquad f,g\in H^{1/2}(\partial D),
\]
and denote its restriction to $\mathcal{X}_{\widetilde\Gamma}(\partial D)\times \mathcal{X}_{\widetilde\Gamma}(\partial D)$ by $\mathcal{M}_{\lambda,\mu}^{\widetilde\Gamma}$, where the trace spaces are defined in Section~\ref{subsec:weak-formulation}.

\Needspace{10\baselineskip}
\begin{theorem}\label{thm:axis}
Let $D$ and $\Omega_D$ be as above, and put $\mu_j(r,\theta,z)=m_j(r,z)$ for $j=1,2$.
\begin{enumerate}
\item Suppose that $D$ is simply connected, $m_j\in L^\infty(D)$, $\lambda_j\in L^\infty(\Omega_D)$, and both pairs satisfy \eqref{eq:convexity}.
If their full torsional forms satisfy
\[
 \mathcal{M}_{\lambda_1,\mu_1}(f,g)=\mathcal{M}_{\lambda_2,\mu_2}(f,g)
 \quad\hbox{for every }f,g\in H^{1/2}(\partial D),
\]
then $\mu_1=\mu_2$ almost everywhere in $\Omega_D$.
For fixed $m$, the entire torsional form is independent of the admissible first parameter $\lambda$, so this conclusion does not determine $\lambda$.
\item Suppose instead that $m_j$ and $\lambda_j$ are smooth in neighborhoods of $\overline{D}$ and $\overline{\Omega_D}$, respectively, and that both pairs satisfy \eqref{eq:convexity}.
Simple connectedness of $D$ is not required in this part.
For any nonempty relatively open arc $\widetilde\Gamma$,
\[
 \mathcal{M}_{\lambda_1,\mu_1}^{\widetilde\Gamma}
 =\mathcal{M}_{\lambda_2,\mu_2}^{\widetilde\Gamma}
 \quad\Longrightarrow\quad \mu_1=\mu_2\ \hbox{on }\overline{\Omega_D}.
\]
If the full elastic maps are equal, then $(\lambda_1,\mu_1)=(\lambda_2,\mu_2)$ on $\overline{\Omega_D}$.
\end{enumerate}
In both parts, the first Lam\'e parameters may be non-axisymmetric.
\end{theorem}

\subsubsection{Separation of variables}\label{subsec:separated-results}
Let $G\subset\mathbb{R}^2$ be a bounded connected $C^\infty$ domain, let $L>0$, and set
\[
 \Omega=G\times(0,L),\qquad x=(y,t).
\]
Here $t$ is a spatial coordinate.
The product domain is Lipschitz, and its boundary is smooth away from the edges $\partial G\times\{0,L\}$.
Fix a nonempty relatively open arc $\Sigma\subset\partial G$ and set
\[
 \Gamma_0=G\times\{0\},\qquad
 \Gamma_\Sigma=\Sigma\times(0,L),\qquad
 \Gamma=\Gamma_0\cup\Gamma_\Sigma.
\]

\Needspace{10\baselineskip}
\begin{theorem}\label{thm:separation}
Let $\Omega$, $\Sigma$, and $\Gamma$ be as above.
Let $(\lambda_j,\mu_j)$, $j=1,2$, be smooth in a common neighborhood of $\overline{\Omega}$ and satisfy \eqref{eq:convexity}.
Suppose that one of the following representations holds for both candidates on $\overline{G}\times[0,L]$:
\begin{enumerate}
\item Multiplicative separation:
\begin{equation}\label{eq:multiplicative-separation}
 \mu_j(y,t)=\alpha_j(y)\beta_j(t),\qquad j=1,2,
\end{equation}
where $\alpha_j\in C^\infty(\overline{G})$ and $\beta_j\in C^\infty([0,L])$ are positive and both are unknown.
\item Additive separation:
\begin{equation}\label{eq:additive-separation}
 \mu_j(y,t)=a_j(y)+b_j(t),\qquad j=1,2,
\end{equation}
where $a_j\in C^\infty(\overline{G})$ and $b_j\in C^\infty([0,L])$ are both unknown.
Neither summand is required to be positive separately; positivity of their sum is included in \eqref{eq:convexity}.
\end{enumerate}
In either case,
\[
 \Lambda_{\lambda_1,\mu_1}^{\Gamma}
   =\Lambda_{\lambda_2,\mu_2}^{\Gamma}
 \quad\Longrightarrow\quad
 \mu_1=\mu_2\quad\hbox{on }\overline{\Omega}.
\]
If the full elastic Dirichlet-to-Neumann maps are equal, then
\[
 (\lambda_1,\mu_1)=(\lambda_2,\mu_2)
 \quad\hbox{on }\overline{\Omega}.
\]
No structural condition is imposed on either $\lambda_j$.
\end{theorem}

\begin{remark}\label{rem:separation-ambiguity}
Under either data equality in Theorem~\ref{thm:separation}, the factors in case~(i) are determined up to $(\alpha,\beta)\mapsto(c\alpha,c^{-1}\beta)$, where $c>0$.
In case~(ii), the summands are determined up to $(a,b)\mapsto(a+c,b-c)$, where $c\in\mathbb{R}$.
The normalizations $\beta(0)=1$ and $b(0)=0$ give individual uniqueness in the respective cases.
\end{remark}

\begin{remark}\label{rem:separation-examples}
The multiplicative representation includes translation-invariant shear moduli when $\beta\equiv1$ and layered shear moduli when $\alpha$ is constant.
More generally, both factors or both summands may vary smoothly with arbitrary oscillation.

The additive class is not a special case of the multiplicative class.
For example, $F(y,t)=2+y_1+t$ is positive on the unit disk times $(0,1)$ and is additively separated.
Every smooth product satisfies $F\partial_{y_1t}F-(\partial_{y_1}F)(\partial_tF)=0$, whereas this expression equals $-1$ for the displayed additive example.
\end{remark}

\subsubsection{Directional quasianalyticity}\label{subsec:quasianalytic-results}
Quasianalyticity preserves the identity principle beyond real-analytic functions, providing a natural route to uniqueness in inverse boundary problems \cite{chen2026quasianalyticity,daude2026local}.
We use the following one-dimensional Denjoy--Carleman class.

\Needspace{14\baselineskip}
\begin{definition}\label{def:quasianalytic}
Let $M=(M_k)_{k\geq0}$ be a nondecreasing positive sequence with $M_0=1$ and
\begin{equation}\label{eq:qa-weight}
 M_k^2\leq M_{k-1}M_{k+1}\quad(k\geq1),
 \qquad
 \sum_{k=0}^{\infty}\frac{M_k}{(k+1)M_{k+1}}=\infty.
\end{equation}
For a connected open interval $J\subset\mathbb{R}$, the \emph{Denjoy--Carleman class of Roumieu type} $C^{\{M\}}(J)$ consists of the functions $q\in C^\infty(J)$ such that, for every compact $K\Subset J$, there are $C,R>0$ with
\[
 \sup_{s\in K}|q^{(k)}(s)|\leq CR^k k!M_k,
 \qquad k\geq0.
\]
The constants may depend on $q$ and $K$, but not on $k$.
\end{definition}

By the Denjoy--Carleman theorem \cite{rainer2014composition,thilliez2008quasianalytic}, this linear class is quasianalytic: for every $q\in C^{\{M\}}(J)$ and $s_*\in J$,
\begin{equation}\label{eq:qa-definition}
 q^{(k)}(s_*)=0\ \hbox{for all integers }k\geq0
 \quad\Longrightarrow\quad q=0\ \hbox{on }J.
\end{equation}

\begin{theorem}\label{thm:qa}
Let $\Omega\subset\mathbb{R}^3$ be bounded, connected, and $C^\infty$.
Let $\lambda_j,\mu_j$, $j=1,2$, be smooth in a common open neighborhood $U$ of $\overline{\Omega}$ and satisfy \eqref{eq:convexity}.
Fix $e\in S^2$ and, for $y\in e^\perp$, define
\begin{equation}\label{eq:qa-line-section}
 I_y=\{s\in\mathbb{R}:y+se\in\Omega\}.
\end{equation}
For each nonempty connected component $(a,b)$ of $I_y$, suppose there is a connected open interval $J\supset[a,b]$, with $y+Je\subset U$, such that the two functions
\[
 s\longmapsto\mu_j(y+se),\qquad s\in J,\quad j=1,2,
\]
belong to a common class $C^{\{M\}}(J)$ as in Definition~\ref{def:quasianalytic}, with $M$ prescribed independently of the candidates.
The weight may vary between components; the defining constants need not be uniform across lines.
Then equality of the full elastic Dirichlet-to-Neumann maps implies
\[
 (\lambda_1,\mu_1)=(\lambda_2,\mu_2)\quad\hbox{on }\overline{\Omega}.
\]
No quasianalyticity or other structural assumption is imposed on $\lambda_j$, and no convexity of $\Omega$ is required.
\end{theorem}

Theorem~\ref{thm:qa} uses quasianalyticity along lines parallel to the fixed direction $e$.
On each connected component of a line section, boundary determination provides all derivatives of $\mu$ at an endpoint, and quasianalyticity determines $\mu$ throughout that component.
Since this argument applies to each component separately, the domain need not be convex.

\begin{remark}\label{rem:curve-qa}
Theorem~\ref{thm:qa} extends to shear moduli that are quasianalytic along a prescribed family of smooth curves.
When these curves start on the observed boundary and cover $\Omega$, localized boundary data determine $\mu$ throughout $\Omega$, while full boundary data determine both Lam\'e parameters; see Proposition~\ref{prop:curve-qa}.
\end{remark}

\begin{remark}
Theorem~\ref{thm:qa} also holds with $C^{\{M\}}(J)$ replaced by any prescribed linear subspace $\mathcal{Q}(J)\subset C^\infty(J)$ satisfying \eqref{eq:qa-definition} for every $q\in\mathcal{Q}(J)$ and $s_*\in J$.
\end{remark}

\begin{remark}\label{rem:analytic-comparison}
Theorem~\ref{thm:qa} extends uniqueness for real-analytic Lam\'e parameters to a strictly larger class of smooth coefficients.
The shear modulus need only be quasianalytic in one fixed direction, allowing classes strictly larger than the real-analytic class \cite{thilliez2008quasianalytic}.
No analyticity is required in the transverse variables or for $\lambda$.
\end{remark}

\subsection{Organization of the paper}\label{subsec:organization}
The rest of the paper is organized as follows.
Section~\ref{sec:preliminaries} gives the weak formulation, boundary determination, and uniqueness for a known shear modulus.
Section~\ref{sec:proof-axis} proves Theorem~\ref{thm:axis} through the torsional reduction to planar conductivity.
Sections~\ref{sec:proof-separation} and~\ref{sec:proof-qa} prove Theorems~\ref{thm:separation} and~\ref{thm:qa}, respectively.

\Needspace{8\baselineskip}
\section{Preliminaries}\label{sec:preliminaries}
This section provides the analytical framework and auxiliary results needed for the proofs of the main theorems.

\subsection{Weak formulation and boundary measurements}\label{subsec:weak-formulation}
Let $\Omega$ be a bounded Lipschitz domain, and let
$\lambda,\mu\in L^\infty(\Omega)$ satisfy \eqref{eq:convexity}.
Define the elastic energy form by
\begin{equation}\label{eq:energy}
 \mathcal{B}_{\lambda,\mu}^{\Omega}(u,v)
 =\int_\Omega\bigl[\lambda(\operatorname{div}u)(\operatorname{div}v)
       +2\mu\varepsilon(u):\varepsilon(v)\bigr]\,dx,
 \qquad u,v\in H^1(\Omega)^3,
\end{equation}
where $A:B=\sum_{i,k=1}^3 A_{ik}B_{ik}$.
We omit the domain superscript when it is clear from the context.
For a Lipschitz domain $U$, we write $\operatorname{Tr}_{\partial U}:H^1(U)\to H^{1/2}(\partial U)$ for the boundary trace operator, applied componentwise to vector fields; thus $u|_{\partial U}$ means $\operatorname{Tr}_{\partial U}u$.

For each $f\in H^{1/2}(\partial\Omega)^3$, the weak solution
$u_f\in H^1(\Omega)^3$ of \eqref{eq:lame-dirichlet} is characterized by
\[
 u_f|_{\partial\Omega}=f,
 \qquad
 \mathcal{B}_{\lambda,\mu}^{\Omega}(u_f,\phi)=0
 \quad\text{for all }\phi\in H_0^1(\Omega)^3.
\]
Existence and uniqueness follow from the trace theorem, Korn's inequality,
and the Lax--Milgram theorem \cite{ciarlet1988mathematical}.

The full Dirichlet-to-Neumann map is defined by
\begin{equation}\label{eq:dn}
 \left\langle\Lambda_{\lambda,\mu}^{\Omega}f,g\right\rangle
 =\mathcal{B}_{\lambda,\mu}^{\Omega}(u_f,V_g),
 \qquad f,g\in H^{1/2}(\partial\Omega)^3,
\end{equation}
where $V_g\in H^1(\Omega)^3$ is any extension of $g$.
The weak equation makes this definition independent of the choice of $V_g$.
For smooth solutions and coefficients, this definition recovers
the traction formula \eqref{eq:elastic-dn-map}.

Let $\Gamma$ be a relatively open subset of the smooth part of
$\partial\Omega$. We consider displacements applied on $\Gamma$
and traction measurements on the same set. Define the trace space for boundary data supported on $\Gamma$ by
\[
 \mathcal{X}_\Gamma(\partial\Omega)
 :=\overline{C_c^\infty(\Gamma)}^{\,H^{1/2}(\partial\Omega)},
\]
where functions in $C_c^\infty(\Gamma)$ are extended by zero
to $\partial\Omega$.
The localized Dirichlet-to-Neumann map is defined by restricting
the full boundary pairing:
\begin{equation}\label{eq:localized-dn-map}
 \left\langle \Lambda_{\lambda,\mu}^{\Gamma}f,g\right\rangle
 :=\left\langle \Lambda_{\lambda,\mu}^{\Omega}f,g\right\rangle,
 \qquad f,g\in\mathcal{X}_\Gamma(\partial\Omega)^3.
\end{equation}
Here $f$ is the prescribed displacement and $g$ tests the resulting
traction. Both vanish almost everywhere on
$\partial\Omega\setminus\Gamma$, where the traction is unmeasured
and need not vanish.

All forms are extended bilinearly to complex arguments.
Equality for real inputs therefore implies equality for complex inputs; the sesquilinear convention used in some references is recovered by conjugating the second argument.

\Needspace{8\baselineskip}
\subsection{Local boundary determination}\label{subsec:boundary-determination}
We establish boundary determination of the Lam\'e parameters and all their Cartesian derivatives from the local boundary maps defined above, using the smooth-boundary reconstruction in \cite{tan2023determining}.
The result applies to smooth boundary patches of both smooth and product domains; if the full boundary maps agree, equality of these derivatives extends to the entire boundary, including the edges of a product domain.

Here and throughout the paper, smoothness on a closed set means restriction of a $C^\infty$ function defined on an open neighborhood of that set.
For smooth Lam\'e parameters satisfying \eqref{eq:convexity} in $\Omega$, the same inequalities hold on $\overline\Omega$ by continuity.

\begin{lemma}\label{lem:boundary-jets}
Let $\Omega$ be a bounded connected domain with $C^\infty$ boundary, or a product domain $G\times(0,L)$ with the geometry specified before Theorem~\ref{thm:separation}.
Let $(\lambda_j,\mu_j)$, $j=1,2$, be smooth in a common neighborhood of $\overline{\Omega}$ and satisfy \eqref{eq:convexity}.
Let $\Gamma$ be a nonempty relatively open subset of the smooth part of $\partial\Omega$.
If $\Lambda_{\lambda_1,\mu_1}^{\Gamma}=\Lambda_{\lambda_2,\mu_2}^{\Gamma}$, then
\[
 D^\alpha(\lambda_1-\lambda_2)(p)
 =D^\alpha(\mu_1-\mu_2)(p)=0
 \quad(p\in \Gamma,\ \alpha\in\mathbb{N}_0^3),
\]
where $\mathbb{N}_0=\{0,1,2,\ldots\}$ and $D^\alpha$ denotes a Cartesian derivative.
If the full elastic maps are equal, these identities hold on all of $\partial\Omega$, including the edges of a product domain.
\end{lemma}

\begin{proof}
\emph{Step 1: localization of the data.}
Let $\Lambda_j=\Lambda_{\lambda_j,\mu_j}^{\Omega}$ denote the full Dirichlet-to-Neumann maps defined in \eqref{eq:dn}, and let $\mathcal L_j=\mathcal L_{\lambda_j,\mu_j}$ be the elastic operators defined in \eqref{eq:lame-dirichlet}.
Fix $p\in \Gamma$, and choose smooth boundary patches $p\in \Gamma_1\Subset \Gamma_2\Subset \Gamma$ contained in one coordinate chart.
Let $\chi\in C_c^\infty(\Gamma_2)$ be real valued and equal to one near $\overline{\Gamma_1}$.
Multiplication by $\chi$ maps $H^{1/2}(\partial\Omega)^3$ into $\mathcal X_{\Gamma}(\partial\Omega)^3$, by approximation inside the smooth chart.
By the definition \eqref{eq:localized-dn-map}, equality of the local maps on $\Gamma$ therefore gives
\begin{equation}\label{eq:localized-map-equality}
 \chi(\Lambda_1-\Lambda_2)\chi=0.
\end{equation}

\emph{Step 2: comparison with a smooth domain.}
The reconstruction used in Step~3 is formulated for smooth domains.
We therefore compare the boundary map on $\Omega$ with that on a smooth domain having the same boundary near the measurement patch.
Choose a bounded smooth domain $\Omega_*\subset\Omega$ and an open neighborhood $U$ of $\overline{\Gamma_2}$ such that $\Omega_*\cap U=\Omega\cap U$.
For a smooth domain, take $\Omega_*=\Omega$; for a product domain, round the edges away from $\overline{\Gamma_2}$.
Fix $j\in\{1,2\}$ and use the same coefficients $(\lambda_j,\mu_j)$ on both domains.
Let $\Lambda_j^*=\Lambda_{\lambda_j,\mu_j}^{\Omega_*}$ be defined by \eqref{eq:dn}.
Boundary functions supported in $\Gamma_2$ are identified on the two boundaries by zero extension.
Define the localized difference of the two boundary maps by
\begin{equation}\label{eq:localized-map-difference}
 R_j:=\chi(\Lambda_j-\Lambda_j^*)\chi.
\end{equation}
We will show that $R_j$ is an integral operator with a smooth matrix-valued kernel.
This comparison allows us to use the smooth-domain maps to recover the local symbol information needed in Step~3.

We first express this difference in terms of a zero Dirichlet problem on $\Omega$.
For $f\in C^\infty(\partial\Omega_*)^3$, let $u_j$ and $u_{j,*}$ solve
\begin{equation}\label{eq:comparison-dirichlet-problems}
 \begin{aligned}
 \mathcal L_ju_j&=0\quad\text{in }\Omega,
 &u_j|_{\partial\Omega}&=\chi f,\\
 \mathcal L_ju_{j,*}&=0\quad\text{in }\Omega_*,
 &u_{j,*}|_{\partial\Omega_*}&=\chi f.
 \end{aligned}
\end{equation}
Choose $\eta\in C_c^\infty(U)$ equal to one in a neighborhood of $\operatorname{supp}\chi$.
Since $\operatorname{supp}\eta\Subset U$ and the two domains coincide in $U$, the product $\eta u_{j,*}$ vanishes near the artificial boundary $\partial\Omega_*\cap\Omega$.
It therefore extends by zero to an $H^1(\Omega)^3$ function, with trace $\eta\chi f=\chi f$.
Thus $z_j:=u_j-\eta u_{j,*}$ has zero boundary trace.
Write $[\mathcal L_j,\eta]v=\mathcal L_j(\eta v)-\eta\mathcal L_jv$.
Using \eqref{eq:comparison-dirichlet-problems}, we obtain
\begin{equation}\label{eq:cutoff-correction-problem}
 \mathcal L_jz_j=-[\mathcal L_j,\eta]u_{j,*}
 \quad\text{in }\Omega,
 \qquad z_j|_{\partial\Omega}=0,
\end{equation}
where the right-hand side is extended by zero from $\Omega_*$.
By the product rule, every term in $[\mathcal L_j,\eta]$ contains a derivative of $\eta$.
Its coefficients therefore vanish near $\operatorname{supp}\chi$.
Since $\eta=1$ there and the two domains have the same outward unit normal on their common boundary, the traction formula \eqref{eq:elastic-dn-map} gives
\begin{equation}\label{eq:localized-correction-traction}
 R_jf=\chi\sigma_{\lambda_j,\mu_j}(z_j)\nu.
\end{equation}
It remains to show that this localized traction is represented by a smooth kernel.

Let $P_j(x,y)$ be the matrix Poisson kernel for $\mathcal L_j$ on the smooth domain $\Omega_*$, so that
\begin{equation}\label{eq:comparison-poisson-representation}
 u_{j,*}(x)=\int_{\partial\Omega_*}P_j(x,y)\chi(y)f(y)\,dS_y.
\end{equation}
On a smooth domain, this kernel is smooth up to the boundary except at pairs with $x=y\in\partial\Omega_*$; see \cite[Chapters~10--11]{grubb2009distributions}.
The commutator in \eqref{eq:cutoff-correction-problem} differentiates only where $\nabla\eta$ is supported.
For its action on \eqref{eq:comparison-poisson-representation} to be nonzero, $x$ must belong to $\operatorname{supp}\nabla\eta$ and $y$ to $\operatorname{supp}\chi$.
These two sets have positive distance because $\eta=1$ near $\operatorname{supp}\chi$.
Here $\mathcal L_{j,x}$ denotes $\mathcal L_j$ acting on the $x$ variable.
Consequently, $[\mathcal L_{j,x},\eta]P_j(x,y)\chi(y)$ is smooth in both variables, extends by zero from $\Omega_*$ to $\Omega$, and vanishes near $\operatorname{supp}\chi$ in the $x$ variable.
For each $y\in\partial\Omega_*$, let $V_j(\cdot,y)\in H_0^1(\Omega)^{3\times3}$ solve
\begin{equation}\label{eq:correction-kernel-problem}
 \mathcal L_{j,x}V_j(x,y)=-[\mathcal L_{j,x},\eta]P_j(x,y)\chi(y)\quad\text{in }\Omega,
 \qquad V_j(\cdot,y)|_{\partial\Omega}=0.
\end{equation}
Equations and tractions involving this matrix are understood columnwise.
The weak Dirichlet solution operator is bounded from $H^{-1}(\Omega)^3$ to $H_0^1(\Omega)^3$.
The right-hand side of \eqref{eq:correction-kernel-problem} depends smoothly on $y$, so $V_j(\cdot,y)$ does too as an $H_0^1$-valued function.
All its $y$-derivatives solve equations whose right-hand sides vanish near $\operatorname{supp}\chi$.
Local boundary estimates applied to these equations show that, for every integer $k\geq0$, the map $y\mapsto V_j(\cdot,y)$ is smooth with values in $C^k$ on a fixed neighborhood of $\operatorname{supp}\chi$ in $\overline\Omega$; see \cite[Chapter~11]{grubb2009distributions}.
Hence $V_j(x,y)$ is jointly smooth near $\operatorname{supp}\chi\times\partial\Omega_*$, including at the common boundary.
By \eqref{eq:comparison-poisson-representation} and \eqref{eq:correction-kernel-problem}, the integral below solves \eqref{eq:cutoff-correction-problem}; uniqueness therefore gives
\begin{equation}\label{eq:correction-integral}
 z_j(x)=\int_{\partial\Omega_*}V_j(x,y)f(y)\,dS_y.
\end{equation}

Substituting \eqref{eq:correction-integral} into the traction identity \eqref{eq:localized-correction-traction} gives
\begin{equation}\label{eq:localized-difference-integral}
 R_jf(x)=\int_{\partial\Omega_*}\chi(x)\sigma_{\lambda_j,\mu_j}\bigl(V_j(\cdot,y)\bigr)(x)\nu(x)f(y)\,dS_y.
\end{equation}
The kernel in \eqref{eq:localized-difference-integral} is smooth by the preceding regularity and is supported in $\operatorname{supp}\chi\times\operatorname{supp}\chi$: if $y\notin\operatorname{supp}\chi$, the right-hand side of \eqref{eq:correction-kernel-problem} vanishes, so $V_j(\cdot,y)=0$.
Thus $R_j$ has a smooth kernel, as claimed.
The argument uses weak solvability on the whole of $\Omega$ and boundary regularity only in $U$, so it also applies when $\Omega$ has edges.

\emph{Step 3: boundary reconstruction from the local symbols.}
Use common boundary normal coordinates $x=F(x',s)$ near $\Gamma_1$, with $s>0$ inside the domain.
Express both boundary maps in the same coordinate frame and surface density; these changes depend only on the known geometry.
The smooth-boundary symbol construction in \cite[Theorem~1.1]{tan2023determining} shows that $\Lambda_j^*$ is a classical pseudodifferential operator of order one.
Combining \eqref{eq:localized-map-equality} with the definition \eqref{eq:localized-map-difference} gives
\begin{equation}\label{eq:localized-symbol-comparison}
 \chi(\Lambda_1^*-\Lambda_2^*)\chi=R_2-R_1,
\end{equation}
whose kernel is smooth by Step~2.
Since $\chi=1$ near $\overline{\Gamma_1}$, \eqref{eq:localized-symbol-comparison} implies that all homogeneous terms of the two full symbols agree on $\Gamma_1$.
Moreover, \eqref{eq:convexity} implies
\[
 \mu_j\geq c_*>0,
 \qquad
 \lambda_j+\mu_j
 =\frac{(3\lambda_j+2\mu_j)+\mu_j}{3}
 \geq\frac{2c_*}{3}>0.
\]
Thus the positivity and smoothness hypotheses of the boundary reconstruction in \cite[Theorem~1.2 and Section~3]{tan2023determining} are satisfied.
Under these hypotheses, the reconstruction determines all normal derivatives of the Lam\'e parameters from the local full symbol.
Equality of the homogeneous symbol terms therefore gives
\begin{equation}\label{eq:boundary-jets}
 \partial_s^k\bigl((q_1-q_2)\circ F\bigr)(x',0)=0,
 \qquad q=\lambda,\mu,\quad k\geq0,
\end{equation}
whenever $F(x',0)\in \Gamma_1$.

\emph{Step 4: Cartesian derivatives and edges.}
The identities \eqref{eq:boundary-jets} hold throughout the patch, so tangential differentiation gives equality of all mixed derivatives in these coordinates.
The chain rule for the fixed smooth chart then gives $D^\alpha(q_1-q_2)=0$ on $\Gamma_1$ for every multi-index $\alpha$ and $q=\lambda,\mu$.
Since $p\in \Gamma$ was arbitrary, this proves the local assertion.
If the full boundary maps are equal, the argument applies on every smooth boundary face.
Every edge point is a limit of points on a smooth face.
Continuity of the ambient Cartesian derivatives therefore extends their vanishing to the edges, proving the assertion on all of $\partial\Omega$.
\end{proof}

\subsection{Uniqueness for a known shear modulus}\label{subsec:known-shear}
We show that equality of the full elastic Dirichlet-to-Neumann maps determines the first Lam\'e parameter once the shear modulus is known.

\begin{lemma}\label{lem:common-extension}
Let $\Omega$ be a bounded connected domain with $C^\infty$ boundary, or a product domain $G\times(0,L)$ with the geometry specified before Theorem~\ref{thm:separation}.
Let $(\lambda_j,\mu_j)$, $j=1,2$, be smooth in a common neighborhood of $\overline\Omega$ and satisfy \eqref{eq:convexity}.
Assume that
\begin{equation}\label{eq:common-extension-assumptions}
 \mu_1=\mu_2\quad\text{in }\Omega,
 \qquad
 \Lambda_{\lambda_1,\mu_1}^{\Omega}
 =\Lambda_{\lambda_2,\mu_2}^{\Omega}.
\end{equation}
Then there exist a ball $B$ with $\overline\Omega\Subset B$ and smooth coefficient pairs $(\widetilde\lambda_j,\widetilde\mu_j)$ on $\overline B$ with the following properties:
\begin{enumerate}
\item Each pair extends the corresponding original coefficients:
\begin{equation}\label{eq:extension-restrictions}
 (\widetilde\lambda_j,\widetilde\mu_j)|_\Omega
 = (\lambda_j,\mu_j),\qquad j=1,2.
\end{equation}
\item The pairs coincide outside $\Omega$ and have the same shear modulus throughout $B$:
\begin{equation}\label{eq:common-exterior-coefficients}
 \begin{aligned}
 (\widetilde\lambda_1,\widetilde\mu_1)
 &=(\widetilde\lambda_2,\widetilde\mu_2)
 &&\text{in }B\setminus\Omega,\\
 \widetilde\mu_1&=\widetilde\mu_2&&\text{in }B.
 \end{aligned}
\end{equation}
\item For some $c_0>0$, both pairs satisfy
\begin{equation}\label{eq:extended-convexity}
 \widetilde\mu_j\geq c_0,\qquad
 3\widetilde\lambda_j+2\widetilde\mu_j\geq c_0
 \quad\text{on }\overline B,\qquad j=1,2.
\end{equation}
Both pairs equal the same constant pair $(\lambda_*,\mu_*)$ in a neighborhood of $\partial B$.
\item The full boundary maps on the ball agree:
\begin{equation}\label{eq:extended-dn-equality}
 \Lambda_{\widetilde\lambda_1,\widetilde\mu_1}^{B}
 =\Lambda_{\widetilde\lambda_2,\widetilde\mu_2}^{B}.
\end{equation}
\end{enumerate}
\end{lemma}

\begin{proof}
We first extend the coefficients so that the two pairs agree outside $\Omega$.
We then glue a solution for the second pair inside $\Omega$ to a solution for the first pair outside it; equality of the original boundary maps makes the glued function a weak solution.

\emph{Step 1: smooth zero extension of the coefficient difference.}
Set $q=\lambda_2-\lambda_1$.
By \eqref{eq:common-extension-assumptions} and Lemma~\ref{lem:boundary-jets},
\begin{equation}\label{eq:flat-jets}
 D^\alpha q=0\quad\text{on }\partial\Omega,
 \qquad \alpha\in\mathbb N_0^3,
\end{equation}
including at the edges of a product domain.
Define the zero extension of $q$ to $\mathbb R^3$ by
\begin{equation}\label{eq:flat-zero-extension}
 Q(x)=
 \begin{cases}
 q(x),&x\in\Omega,\\
 0,&x\notin\Omega.
 \end{cases}
\end{equation}
For every multi-index $\alpha$, the zero extension of $D^\alpha q$ is differentiable across $\partial\Omega$: its value and first derivatives vanish there by \eqref{eq:flat-jets}, so Taylor's formula gives an $O(|h|^2)$ bound at each boundary point.
Its derivative is the zero extension of the corresponding derivative of $D^\alpha q$.
Induction therefore gives $Q\in C_c^\infty(\mathbb R^3)$, including across the edges, because \eqref{eq:flat-jets} concerns ambient Cartesian derivatives.

\emph{Step 2: construction of the two coefficient extensions.}
By continuity and \eqref{eq:convexity}, choose a bounded neighborhood $U_0$ of $\overline\Omega$ within the common neighborhood of smoothness, on which
\[
 \mu_1\geq c_*/2,\qquad 3\lambda_1+2\mu_1\geq c_*/2.
\]
Choose $\chi\in C_c^\infty(U_0)$ with $0\leq\chi\leq1$ and $\chi=1$ near $\overline\Omega$, and constants $(\lambda_*,\mu_*)$ satisfying $\mu_*>0$ and $3\lambda_*+2\mu_*>0$.
Using $Q$ from \eqref{eq:flat-zero-extension}, define the two extended Lam\'e coefficient pairs on $\mathbb R^3$ by
\begin{equation}\label{eq:common-extension-construction}
 \begin{aligned}
 (\widetilde\lambda_1,\widetilde\mu_1)
 &=\chi(\lambda_1,\mu_1)+(1-\chi)(\lambda_*,\mu_*),\\
 (\widetilde\lambda_2,\widetilde\mu_2)
 &=(\widetilde\lambda_1+Q,\widetilde\mu_1),
 \end{aligned}
\end{equation}
where the products with $\chi$ are extended by zero outside $U_0$.
The pairs are smooth because $Q\in C_c^\infty(\mathbb R^3)$ by Step~1.
Since $\chi=1$ on $\Omega$, $Q=\lambda_2-\lambda_1$ there, and $Q=0$ outside, \eqref{eq:common-extension-construction} gives \eqref{eq:extension-restrictions}--\eqref{eq:common-exterior-coefficients}.
The first pair is a convex combination of strongly convex pairs, with its shear modulus and $3\widetilde\lambda_1+2\widetilde\mu_1$ bounded below by
\[
 c_0:=\min\{c_*/2,\mu_*,3\lambda_*+2\mu_*\}>0.
\]
The second pair equals the original second pair in $\Omega$ and the first pair outside, so it has the same lower bounds.
Choose $B$ with $\overline\Omega\cup\operatorname{supp}\chi\Subset B$; both pairs equal $(\lambda_*,\mu_*)$ near $\partial B$.
This proves properties~\textup{(i)--(iii)}.

\emph{Step 3: construction of the second extended solution.}
To prove \eqref{eq:extended-dn-equality}, we will construct solutions with the same boundary input on $\partial B$ that agree outside $\Omega$.
We use the weak-solution gluing argument given for scalar conductivity equations in \cite[Section~2]{astala2006calderon}, and verify it here for the elastic energy form.
Fix $F\in H^{1/2}(\partial B)^3$ and let $U_1\in H^1(B)^3$ solve the elastic Dirichlet problem \eqref{eq:lame-dirichlet} for the first extended pair:
\begin{equation}\label{eq:extended-first-problem}
 \mathcal L_{\widetilde\lambda_1,\widetilde\mu_1}U_1=0
 \quad\text{in }B,\qquad U_1|_{\partial B}=F.
\end{equation}
Set $f=\operatorname{Tr}_{\partial\Omega}(U_1|_\Omega)$.
By \eqref{eq:extension-restrictions}, testing \eqref{eq:extended-first-problem} with zero extensions of functions in $H_0^1(\Omega)^3$ shows that $U_1|_\Omega$ solves the first original equation with trace $f$.
Let $v_2\in H^1(\Omega)^3$ solve the second original problem with the same trace:
\begin{equation}\label{eq:interior-second-problem}
 \mathcal L_{\lambda_2,\mu_2}v_2=0
 \quad\text{in }\Omega,\qquad v_2|_{\partial\Omega}=f.
\end{equation}
Define the second displacement on $B$ by gluing $v_2$ to $U_1$ across $\partial\Omega$:
\begin{equation}\label{eq:glued-displacement}
 U_2:=\begin{cases}
 v_2,&\text{in }\Omega,\\
 U_1,&\text{in }B\setminus\overline\Omega.
 \end{cases}
\end{equation}
The difference $v_2-U_1|_\Omega$ has zero trace by \eqref{eq:interior-second-problem}, so its zero extension to $B$ belongs to $H_0^1(B)^3$.
Thus $U_2\in H^1(B)^3$ has trace $F$ on $\partial B$, and its weak gradient is the corresponding piecewise gradient.
It remains to check that no interface force appears in its weak equation.

For $\Phi\in H_0^1(B)^3$, split the energy form \eqref{eq:energy} over $\Omega$ and its exterior, using \eqref{eq:extension-restrictions}--\eqref{eq:common-exterior-coefficients} and \eqref{eq:glued-displacement}.
The first extended solution has zero energy pairing with $\Phi$ by \eqref{eq:extended-first-problem}; the two interior pairings are the boundary pairings in \eqref{eq:dn}.
Hence
\begin{equation}\label{eq:glued-weak-equation}
 \mathcal B_{\widetilde\lambda_2,\widetilde\mu_2}^B(U_2,\Phi)
 =\left\langle
 (\Lambda_{\lambda_2,\mu_2}^\Omega
 -\Lambda_{\lambda_1,\mu_1}^\Omega)f,
 \operatorname{Tr}_{\partial\Omega}(\Phi|_\Omega)
 \right\rangle
 =0.
\end{equation}
The last equality follows from \eqref{eq:common-extension-assumptions}.
Thus \eqref{eq:glued-weak-equation} proves that $U_2$ is the weak solution for the second extended pair with boundary value $F$.

\emph{Step 4: equality of the boundary maps on the ball.}
By \eqref{eq:common-exterior-coefficients} and \eqref{eq:glued-displacement}, the coefficients and solutions agree outside $\Omega$, where we can compare their outer boundary pairings.
Given $H\in H^{1/2}(\partial B)^3$, choose an extension $W\in H^1(B)^3$ with trace $H$ that vanishes near $\overline\Omega$; a cutoff of any trace extension provides one.
The energy integrands agree on the support of $W$, so \eqref{eq:dn} gives
\begin{equation}\label{eq:extended-dn-pairing}
 \begin{aligned}
 \left\langle\Lambda_{\widetilde\lambda_2,\widetilde\mu_2}^B F,H\right\rangle
 &=\mathcal B_{\widetilde\lambda_2,\widetilde\mu_2}^B(U_2,W)\\
 &=\mathcal B_{\widetilde\lambda_1,\widetilde\mu_1}^B(U_1,W)
 =\left\langle\Lambda_{\widetilde\lambda_1,\widetilde\mu_1}^B F,H\right\rangle.
 \end{aligned}
\end{equation}
Since $F$ and $H$ are arbitrary, \eqref{eq:extended-dn-pairing} proves \eqref{eq:extended-dn-equality}.
The argument uses only weak solutions and does not require $B\setminus\overline\Omega$ to be connected.
\end{proof}

The weighted estimate gives uniqueness of the first Lam\'e parameter when the shear modulus is known \cite{eskin2002inverse,eskin2003inverse}.

\begin{proposition}\label{prop:known-shear}
Let $\Omega$ be a bounded connected domain with $C^\infty$ boundary, or a product domain $G\times(0,L)$ with the geometry specified before Theorem~\ref{thm:separation}.
Let $(\lambda_j,\mu_j)$, $j=1,2$, be smooth in a common neighborhood of $\overline{\Omega}$ and satisfy \eqref{eq:convexity}.
If
\[
 \Lambda_{\lambda_1,\mu_1}^{\Omega}
 =\Lambda_{\lambda_2,\mu_2}^{\Omega}
 \quad\hbox{and}\quad
 \mu_1=\mu_2\quad\hbox{in }\Omega,
\]
then $\lambda_1=\lambda_2$ on $\overline{\Omega}$.
\end{proposition}

\begin{proof}
Lemma~\ref{lem:common-extension} gives smooth strongly convex extensions to a ball $B$ with a common shear modulus and equal full boundary maps.
Both extended pairs equal the same constant pair near $\partial B$.
For either extended pair, strong convexity gives
\[
 \widetilde\lambda+\widetilde\mu
 =\frac{3\widetilde\lambda+2\widetilde\mu+\widetilde\mu}{3}>0,
 \qquad
 \widetilde\lambda+2\widetilde\mu
 =\frac{3\widetilde\lambda+2\widetilde\mu+4\widetilde\mu}{3}>0.
\]
Together with $\widetilde\mu>0$, these are the positivity conditions required by the elasticity reduction.
For these extensions, the weighted estimate in \cite[Theorem~2]{eskin2002inverse} (see also \cite[Theorem~3.1]{eskin2003inverse}) gives constants $C,\tau_0>0$ such that
\begin{equation}\label{eq:er-estimate}
 \|\widetilde\lambda_1+\widetilde\mu_1
        -\widetilde\lambda_2-\widetilde\mu_2\|_\tau
 \leq\frac C\tau\|\widetilde\mu_1-\widetilde\mu_2\|_\tau,
 \qquad \tau>\tau_0,
\end{equation}
where $C$ is independent of $\tau$.
Both $C$ and $\tau_0$ may depend on the extended coefficients and $B$.
The weighted norm is
\[
 \|q\|_\tau^2=\int_B |q(x)|^2 e^{2\tau|x|^2}\,dx.
\]
Estimate \eqref{eq:er-estimate} requires no smallness of $\nabla\widetilde\mu$.
The gradient restriction in \cite[Theorem~1]{eskin2002inverse} belongs to a separate perturbative uniqueness result.

Fix $\tau>\tau_0$.
Using the common shear modulus in \eqref{eq:er-estimate}, we obtain
\[
 \|\widetilde\lambda_1-\widetilde\lambda_2\|_\tau
 \leq\frac C\tau
       \|\widetilde\mu_1-\widetilde\mu_2\|_\tau=0.
\]
Positivity of the weight gives equality almost everywhere in $B$.
Smoothness then yields $\lambda_1=\lambda_2$ on $\overline{\Omega}$.
\end{proof}

\section{Axisymmetric shear moduli}\label{sec:proof-axis}
We prove Theorem~\ref{thm:axis} by reducing torsional measurements of the three-dimensional elastic system to a planar conductivity problem.
This reduction determines the axisymmetric shear modulus.
In the smooth case, the full elastic boundary map then determines the first Lam\'e parameter.

\Needspace{16\baselineskip}
\subsection{The exact torsional reduction}\label{subsec:torsional-reduction}
We recall the classical planar conductivity problem and show its connection to torsional measurements in elasticity.
Let $D\subset\mathbb{R}^2$ be a bounded connected domain with $C^\infty$ boundary.
For a real conductivity $\gamma\in L^\infty(D)$ satisfying $0<c_\gamma\leq\gamma\leq C_\gamma<\infty$ almost everywhere, let $v_f\in H^1(D)$ solve
\begin{equation}\label{eq:planar-conductivity-dirichlet}
 \operatorname{div}(\gamma\nabla v_f)=0\quad\hbox{in }D,
 \qquad v_f|_{\partial D}=f\in H^{1/2}(\partial D).
\end{equation}
The scalar Dirichlet-to-Neumann map $\Lambda^c_\gamma:H^{1/2}(\partial D)\to H^{-1/2}(\partial D)$ is defined by
\begin{equation}\label{eq:planar-conductivity-dn}
 \left\langle \Lambda^c_\gamma f,g\right\rangle
 =\int_D\gamma\nabla v_f\cdot\nabla w_g\,dx,
 \qquad f,g\in H^{1/2}(\partial D),
\end{equation}
where $w_g\in H^1(D)$ is any extension of $g$.

The following proposition shows that, for the torsional data \eqref{eq:torsional-boundary-displacement}, the elastic problem \eqref{eq:lame-dirichlet} reduces to the planar conductivity problem \eqref{eq:planar-conductivity-dirichlet}, with equivalent boundary measurements \eqref{eq:dn} and \eqref{eq:planar-conductivity-dn}.

\begin{proposition}\label{prop:torsion}
Let $D$ and $\Omega_D$ be as in Theorem~\ref{thm:axis}.
Let $m\in L^\infty(D)$ be bounded below by a positive constant and set $\mu(r,\theta,z)=m(r,z)$.
Let $\lambda\in L^\infty(\Omega_D)$ satisfy \eqref{eq:convexity} together with $\mu$.
Set $\gamma(r,z)=r^3m(r,z)$.
The torsional boundary lift $\mathcal{T}$ in \eqref{eq:torsional-boundary-displacement} is bounded from $H^{1/2}(\partial D)$ to $H^{1/2}(\partial\Omega_D)^3$.
For every nonempty relatively open arc $\widetilde\Gamma\subset\partial D$, let $\Gamma\subset\partial\Omega_D$ be its surface of revolution; then
\[
 \mathcal{T}\bigl(\mathcal{X}_{\widetilde\Gamma}(\partial D)\bigr)
 \subset\mathcal{X}_\Gamma(\partial\Omega_D)^3.
\]
For $f\in H^{1/2}(\partial D)$, let $v_f$ solve \eqref{eq:planar-conductivity-dirichlet} with conductivity $\gamma$.
Then $u_{\mathcal{T}f}=rv_f e_\theta$, where $u_{\mathcal{T}f}$ is the weak solution of \eqref{eq:lame-dirichlet} with boundary value $\mathcal{T}f$.
Moreover, the boundary pairings \eqref{eq:dn} and \eqref{eq:planar-conductivity-dn} satisfy
\begin{equation}\label{eq:torsion-dn}
 \left\langle \Lambda_{\lambda,\mu}^{\Omega_D}\mathcal{T}f,\mathcal{T}g\right\rangle
 =2\pi\left\langle \Lambda^c_\gamma f,g\right\rangle,
 \qquad f,g\in H^{1/2}(\partial D).
\end{equation}
\end{proposition}

\begin{proof}
\emph{Step 1: the torsional lift and its boundary trace.}
We first lift scalar fields on $D$ to torsional displacements on $\Omega_D$ to make the boundary formula \eqref{eq:torsional-boundary-displacement} precise for nonsmooth data.
The same lift will be applied to the planar solution $v_f$ of \eqref{eq:planar-conductivity-dirichlet}.
Write $e_r=(\cos\theta,\sin\theta,0)$ and $e_z=(0,0,1)$.
Define the interior torsional lift $\mathcal{E}:H^1(D)\to H^1(\Omega_D)^3$ by $\mathcal{E}v=rv e_\theta$.
For smooth $v$, direct differentiation gives
\[
 \partial_r(\mathcal{E}v)=(v+r\partial_rv)e_\theta,\qquad
 \partial_z(\mathcal{E}v)=r\partial_zv\,e_\theta,\qquad
 r^{-1}\partial_\theta(\mathcal{E}v)=-v e_r.
\]
Using the cylindrical volume element $r\,dr\,d\theta\,dz$, we obtain
\begin{equation}\label{eq:torsional-lift-bound}
 \begin{aligned}
 \|\mathcal{E}v\|_{H^1(\Omega_D)^3}^2
 =2\pi\int_D r\bigl(&r^2|v|^2+|v+r\partial_rv|^2\\
                    &+r^2|\partial_zv|^2+|v|^2\bigr)\,dr\,dz
 \leq C_D\|v\|_{H^1(D)}^2.
 \end{aligned}
\end{equation}
Here we use $0<r_-\leq r\leq r_+$ on $\overline D$.
By density, \eqref{eq:torsional-lift-bound} proves the asserted boundedness of $\mathcal{E}$.
Since $\mathcal{E}(C_c^\infty(D))\subset C_c^\infty(\Omega_D)^3$, density also gives
\[
 \mathcal{E}\bigl(H_0^1(D)\bigr)\subset H_0^1(\Omega_D)^3.
\]

We now make the boundary lift \eqref{eq:torsional-boundary-displacement} precise for nonsmooth data using the trace of $\mathcal{E}v$.
For $f\in H^{1/2}(\partial D)$, choose any $v\in H^1(D)$ with $\operatorname{Tr}_{\partial D}v=f$ and set
\begin{equation}\label{eq:torsional-trace-map}
 \mathcal{T}:H^{1/2}(\partial D)\to H^{1/2}(\partial\Omega_D)^3,
 \qquad \mathcal{T}f=\operatorname{Tr}_{\partial\Omega_D}(\mathcal{E}v).
\end{equation}
Any two extensions of $f$ differ by an element of $H_0^1(D)$, whose image under $\mathcal{E}$ has zero boundary trace by the preceding inclusion.
Thus $\mathcal{T}f$ is independent of the chosen extension.

By the trace theorem \cite{ciarlet1988mathematical}, a bounded right inverse of the trace map provides an extension $v$ of $f$ with $\|v\|_{H^1(D)}\leq C\|f\|_{H^{1/2}(\partial D)}$.
Using this choice in \eqref{eq:torsional-trace-map}, the trace theorem and \eqref{eq:torsional-lift-bound} give
\[
 \|\mathcal{T}f\|_{H^{1/2}(\partial\Omega_D)^3}
 \leq C_D\|\mathcal{E}v\|_{H^1(\Omega_D)^3}
 \leq C_D'\|f\|_{H^{1/2}(\partial D)}.
\]
For smooth $f$, we may choose a smooth $v$ and obtain $\mathcal{T}f=rf e_\theta$, as in \eqref{eq:torsional-boundary-displacement}.
The estimate extends this formula to all $H^{1/2}$ boundary data.

Finally, if $f\in C_c^\infty(\widetilde\Gamma)$, then
$\mathcal{T}f\in C_c^\infty(\Gamma)^3$ because $\Gamma$ is the surface swept out by
$\widetilde\Gamma$.
By the boundedness of $\mathcal{T}$ and the definition of the localized trace spaces in
Section~\ref{subsec:weak-formulation},
\[
 \mathcal{T}\bigl(\mathcal{X}_{\widetilde\Gamma}(\partial D)\bigr)
 \subset\mathcal{X}_{\Gamma}(\partial\Omega_D)^3.
\]

\emph{Step 2: converting vector tests to planar tests.}
By Step~1, each $v\in H^1(D)$ has the torsional lift $u=\mathcal{E}v=rv e_\theta$.
Our aim is to express the elastic energy $\mathcal{B}_{\lambda,\mu}^{\Omega_D}(u,\Phi)$ for any three-dimensional vector test $\Phi$ as the planar conductivity energy of $v$ against a scalar test on $D$.
We first identify which components of $\Phi$ can contribute to this energy.

In the orthonormal cylindrical frame, subscripts $i\in\{r,\theta,z\}$ denote directional components, such as $u_i=u\cdot e_i$ and $\Phi_i=\Phi\cdot e_i$.
For the strain tensor $\varepsilon$ in \eqref{eq:strain-stress}, write $\varepsilon_{ij}(u)=e_i\cdot\varepsilon(u)e_j$, and similarly for $\Phi$.
For smooth $v$, the lifted field has $u_r=u_z=0$ and $u_\theta=rv(r,z)$; the last component can still vary with $r$ and $z$.
Since $u_\theta$ is independent of $\theta$, the cylindrical divergence formula gives
\[
 \operatorname{div}u
 =\frac1r\partial_r(ru_r)+\frac1r\partial_\theta u_\theta+\partial_z u_z
 =\frac1r\partial_\theta(rv)=0.
\]
The only possibly nonzero strain components are
\begin{equation}\label{eq:torsional-strain-components}
 \varepsilon_{r\theta}(u)=\varepsilon_{\theta r}(u)=\frac r2\partial_rv,
 \qquad
 \varepsilon_{z\theta}(u)=\varepsilon_{\theta z}(u)=\frac r2\partial_zv.
\end{equation}
Approximation extends these identities to every $v\in H^1(D)$.
The divergence calculation removes the $\lambda$ term from \eqref{eq:energy}, even when $\lambda$ depends on $\theta$.
By \eqref{eq:energy} and \eqref{eq:torsional-strain-components}, the remaining pairing with a vector test $\Phi$ involves only its $r\theta$ and $z\theta$ strains.
To see how this pairing can become a planar one, we compute those strains for $\Phi\in C_c^\infty(\Omega_D)^3$:
\[
 \begin{aligned}
 \varepsilon_{r\theta}(\Phi)
 &=\frac12\left(\partial_r\Phi_\theta-\frac{\Phi_\theta}{r}
                           +\frac1r\partial_\theta\Phi_r\right),\\
 \varepsilon_{z\theta}(\Phi)
 &=\frac12\left(\partial_z\Phi_\theta
                           +\frac1r\partial_\theta\Phi_z\right).
 \end{aligned}
\]
The $\theta$-derivatives integrate to zero over a full turn, while
\[
 \partial_r\Phi_\theta-\frac{\Phi_\theta}{r}
 =r\partial_r\left(\frac{\Phi_\theta}{r}\right),
 \qquad
 \partial_z\Phi_\theta
 =r\partial_z\left(\frac{\Phi_\theta}{r}\right).
\]
This is why we define the angular averaging operator $\mathcal{P}:C_c^\infty(\Omega_D)^3\to C_c^\infty(D)$ by
\[
 \mathcal{P}\Phi(r,z)=\frac1{2\pi}\int_0^{2\pi}
             \frac{\Phi_\theta(r,\theta,z)}{r}\,d\theta.
\]
The preceding identities give
\begin{equation}\label{eq:torsional-averaged-strains}
 \frac1{2\pi}\int_0^{2\pi}\varepsilon_{r\theta}(\Phi)\,d\theta
 =\frac r2\partial_r(\mathcal{P}\Phi),\qquad
 \frac1{2\pi}\int_0^{2\pi}\varepsilon_{z\theta}(\Phi)\,d\theta
 =\frac r2\partial_z(\mathcal{P}\Phi).
\end{equation}
Since $r$ stays bounded above and away from zero, differentiation under the integral and Cauchy--Schwarz give
\begin{equation}\label{eq:torsional-average-bound}
 \|\mathcal{P}\Phi\|_{H^1(D)}\leq C_D\|\Phi\|_{H^1(\Omega_D)^3}.
\end{equation}
The projection of $\operatorname{supp}\Phi$ is compact in $D$, so $\mathcal{P}\Phi\in C_c^\infty(D)$.
By density, $\mathcal{P}$ extends to a bounded map $H_0^1(\Omega_D)^3\to H_0^1(D)$.

\emph{Step 3: the elastic equation.}
For smooth $v$ and $\Phi\in C_c^\infty(\Omega_D)^3$, apply the energy form \eqref{eq:energy} to $\mathcal{E}v$ and $\Phi$.
Each off-diagonal strain entry occurs twice in the Frobenius product.
Using \eqref{eq:torsional-strain-components} and \eqref{eq:torsional-averaged-strains}, together with $\mu=m(r,z)$ and $dx=r\,dr\,d\theta\,dz$, we obtain
\begin{equation}\label{eq:torsional-energy-reduction}
 \begin{aligned}
 \mathcal{B}_{\lambda,\mu}^{\Omega_D}(\mathcal{E}v,\Phi)
 &=4\int_D\int_0^{2\pi}mr
   \left[\frac r2\partial_rv\,\varepsilon_{r\theta}(\Phi)
        +\frac r2\partial_zv\,\varepsilon_{z\theta}(\Phi)\right]
      d\theta\,dr\,dz\\
 &=2\pi\int_D\gamma\nabla v\cdot\nabla(\mathcal{P}\Phi)\,dr\,dz.
 \end{aligned}
\end{equation}
The estimates \eqref{eq:torsional-lift-bound} and \eqref{eq:torsional-average-bound}, together with the coefficient assumptions of Proposition~\ref{prop:torsion}, show that both sides of \eqref{eq:torsional-energy-reduction} are continuous bilinear forms on $H^1(D)\times H_0^1(\Omega_D)^3$.
Since $C^\infty(\overline D)$ is dense in $H^1(D)$ and $C_c^\infty(\Omega_D)^3$ is dense in $H_0^1(\Omega_D)^3$, the identity holds for all $v$ and $\Phi$ in these spaces.

For the planar solution $v_f$ of \eqref{eq:planar-conductivity-dirichlet} and any $\Phi\in H_0^1(\Omega_D)^3$, Step~2 gives $\mathcal{P}\Phi\in H_0^1(D)$.
We may therefore test the planar weak equation with $\mathcal{P}\Phi$.
The extended identity \eqref{eq:torsional-energy-reduction} then gives
\[
 \mathcal{B}_{\lambda,\mu}^{\Omega_D}(\mathcal{E}v_f,\Phi)
 =2\pi\int_D\gamma\nabla v_f\cdot\nabla(\mathcal{P}\Phi)\,dr\,dz=0.
\]
Also, $\operatorname{Tr}_{\partial D}v_f=f$, so the boundary-lift identity \eqref{eq:torsional-trace-map} gives
\[
 \operatorname{Tr}_{\partial\Omega_D}(\mathcal{E}v_f)=\mathcal{T}f.
\]
Thus $\mathcal{E}v_f$ satisfies the weak formulation of \eqref{eq:lame-dirichlet} with boundary value $\mathcal{T}f$.
Uniqueness gives $\mathcal{E}v_f=u_{\mathcal{T}f}$.

\emph{Step 4: the boundary pairing.}
Let $w\in H^1(D)$ have trace $g$.
By \eqref{eq:torsional-trace-map}, $\operatorname{Tr}_{\partial\Omega_D}(\mathcal{E}w)=\mathcal{T}g$.
The strain formula \eqref{eq:torsional-strain-components} gives
\[
 2m\varepsilon(\mathcal{E}v_f):\varepsilon(\mathcal{E}w)
 =mr^2\nabla v_f\cdot\nabla w.
\]
Using \eqref{eq:dn}, \eqref{eq:planar-conductivity-dn}, and $dx=r\,dr\,d\theta\,dz$, we obtain
\[
 \begin{aligned}
 \left\langle\Lambda_{\lambda,\mu}^{\Omega_D}\mathcal{T}f,\mathcal{T}g\right\rangle
 &=\mathcal{B}_{\lambda,\mu}^{\Omega_D}(\mathcal{E}v_f,\mathcal{E}w)\\
 &=2\pi\int_D\gamma\nabla v_f\cdot\nabla w\,dr\,dz
 =2\pi\left\langle\Lambda^c_\gamma f,g\right\rangle.
 \end{aligned}
\]
This proves \eqref{eq:torsion-dn}.
\end{proof}

\subsection{Planar conductivity uniqueness}\label{subsec:planar-conductivity}
The identity \eqref{eq:torsion-dn} converts torsional elastic measurements into the planar conductivity measurements defined in \eqref{eq:planar-conductivity-dn}.
We now recall the corresponding uniqueness results for full and partial boundary data in Proposition~\ref{prop:planar-uniqueness}.

\begin{proposition}\label{prop:planar-uniqueness}
Let $D\subset\mathbb{R}^2$ be a bounded connected domain with $C^\infty$ boundary, and let $\gamma_j\in L^\infty(D)$, $j=1,2$, be real valued, with
\[
 0<c\leq\gamma_j\leq C<\infty
 \quad\hbox{almost everywhere in }D.
\]
\begin{enumerate}
\item If $D$ is simply connected, then
\[
 \Lambda^c_{\gamma_1}=\Lambda^c_{\gamma_2}
 \quad\Longrightarrow\quad
 \gamma_1=\gamma_2\quad\hbox{almost everywhere in }D.
\]
\item Suppose that $\gamma_j\in C^\infty(\overline D)$ and let $\widetilde\Gamma\subset\partial D$ be any nonempty relatively open arc.
If
\[
 \left\langle(\Lambda^c_{\gamma_1}-\Lambda^c_{\gamma_2})f,g\right\rangle=0
 \quad\hbox{for all }f,g\in\mathcal{X}_{\widetilde\Gamma}(\partial D),
\]
then $\gamma_1=\gamma_2$ on $\overline D$.
No simple connectedness assumption is required in this part.
\end{enumerate}
\end{proposition}

\begin{proof}
Part~(i) is the planar conductivity uniqueness theorem in \cite[Theorem~1]{astala2006calderon}.
For part~(ii), take $f\in H^{1/2}(\partial D)$ with $\operatorname{supp}f\Subset\widetilde\Gamma$.
By density of smooth data supported on the arc and boundedness of the scalar maps, the assumed pairing identity implies
\[
 (\Lambda^c_{\gamma_1}f)|_{\widetilde\Gamma}
 =(\Lambda^c_{\gamma_2}f)|_{\widetilde\Gamma}
 \quad\hbox{in }\mathcal{D}'(\widetilde\Gamma).
\]
The uniqueness result for inputs and measurements on the same arc \cite[Corollary~1.1]{imanuvilov2010calderon} therefore gives $\gamma_1=\gamma_2$ in $D$, and continuity extends the equality to $\overline D$.
\end{proof}

\subsection{Completion of the proof}
\begin{proof}[Proof of Theorem~\ref{thm:axis}]
We first prove part~(i).
Let the two coefficient pairs satisfy its hypotheses, and assume that their full torsional forms agree.
The axisymmetric representation $\mu_j(r,\theta,z)=m_j(r,z)$, the boundedness of the coefficients, and \eqref{eq:convexity} allow us to apply Proposition~\ref{prop:torsion} to both pairs.
For the associated planar conductivities $\gamma_j(r,z)=r^3m_j(r,z)$, the boundary pairing identity \eqref{eq:torsion-dn} gives, for every $f,g\in H^{1/2}(\partial D)$,
\begin{equation}\label{eq:axis-measurement-reduction}
 2\pi\left\langle(\Lambda^c_{\gamma_1}-\Lambda^c_{\gamma_2})f,g\right\rangle
 =\mathcal{M}_{\lambda_1,\mu_1}(f,g)
  -\mathcal{M}_{\lambda_2,\mu_2}(f,g).
\end{equation}
The assumed data equality makes the right-hand side vanish for every $f,g\in H^{1/2}(\partial D)$, so $\Lambda^c_{\gamma_1}=\Lambda^c_{\gamma_2}$.
By $m_j\in L^\infty(D)$ and \eqref{eq:convexity}, the functions $m_j$ are bounded above and below by positive constants.
Since $0<r_-\leq r\leq r_+$, the same holds for $\gamma_j$.
The assumed simple connectedness of $D$ therefore allows us to apply Proposition~\ref{prop:planar-uniqueness}(i), which yields $\gamma_1=\gamma_2$ almost everywhere in $D$.
Dividing by $r^3>0$ gives $m_1=m_2$ almost everywhere in $D$, and the cylindrical change of variables $dx=r\,dr\,d\theta\,dz$ then gives $\mu_1=\mu_2$ almost everywhere in $\Omega_D$.
For fixed $m$, \eqref{eq:torsion-dn} expresses the entire torsional form in terms of $\gamma=r^3m$, independently of the admissible choice of $\lambda$.
This also proves the assertion that torsional data do not determine $\lambda$.

For part~(ii), assume that the coefficient pairs satisfy the stated smoothness and convexity conditions and that their localized torsional forms agree on the given arc $\widetilde\Gamma$.
Proposition~\ref{prop:torsion} again gives \eqref{eq:axis-measurement-reduction}, whose right-hand side now vanishes for $f,g\in\mathcal{X}_{\widetilde\Gamma}(\partial D)$.
The conductivities $\gamma_j=r^3m_j$ are smooth on $\overline D$ and have the positive bounds established above.
Thus the resulting conductivity pairing identity satisfies all hypotheses of Proposition~\ref{prop:planar-uniqueness}(ii), which yields $\gamma_1=\gamma_2$ on $\overline D$ without requiring simple connectedness.
Dividing by $r^3>0$ now gives $\mu_1=\mu_2$ on $\overline{\Omega_D}$.

Finally, suppose that the full elastic maps agree in part~(ii).
Evaluating these maps on the torsional boundary data \eqref{eq:torsional-boundary-displacement} gives equality of the full torsional forms and hence of their localized restrictions.
The preceding argument therefore gives a common shear modulus.
The geometry of $D$ makes $\Omega_D$ bounded, connected, and smooth; since $r$ stays away from zero, the functions $\mu_j=m_j(r,z)$ are also smooth near $\overline{\Omega_D}$.
Together with the assumed smoothness of $\lambda_j$ and \eqref{eq:convexity}, this verifies the hypotheses of Proposition~\ref{prop:known-shear}, which yields $\lambda_1=\lambda_2$ on $\overline{\Omega_D}$.
\end{proof}

\section{Separation of variables}\label{sec:proof-separation}
To prove Theorem~\ref{thm:separation}, we first use the localized measurements \eqref{eq:localized-dn-map} to determine $\mu$ on the bottom and observed lateral faces.
The separated forms \eqref{eq:multiplicative-separation} and \eqref{eq:additive-separation} then determine $\mu$ throughout $\Omega$.
With full boundary data, Proposition~\ref{prop:known-shear} then determines $\lambda$ from the common shear modulus.

Recall that $\Omega=G\times(0,L)$, where $G\subset\mathbb{R}^2$ is a bounded connected domain with smooth boundary, and $\Sigma\subset\partial G$ is a nonempty relatively open arc.
The observed sets in Theorem~\ref{thm:separation} are the bottom face $\Gamma_0=G\times\{0\}$ and the lateral patch $\Gamma_\Sigma=\Sigma\times(0,L)$.

\begin{proof}[Proof of Theorem~\ref{thm:separation}]
Let the two coefficient pairs be smooth in a common neighborhood of $\overline\Omega$, satisfy \eqref{eq:convexity}, and admit one of the representations \eqref{eq:multiplicative-separation} and \eqref{eq:additive-separation}, as assumed in Theorem~\ref{thm:separation}.
Suppose that their localized maps defined in \eqref{eq:localized-dn-map} agree on $\Gamma=\Gamma_0\cup\Gamma_\Sigma$.
Restricting these boundary pairings to data supported on either face preserves the equality.
The product geometry and the coefficient assumptions allow us to apply Lemma~\ref{lem:boundary-jets} on both faces.
Since $(y,0)\in\Gamma_0$ for every $y\in G$ and $(y,t)\in\Gamma_\Sigma$ for every $y\in\Sigma$ and $t\in(0,L)$, we obtain
\begin{equation}\label{eq:separated-traces}
 \begin{aligned}
 \mu_1(y,0)&=\mu_2(y,0),&&y\in G,\\
 \mu_1(y,t)&=\mu_2(y,t),&&y\in\Sigma,\quad 0<t<L.
 \end{aligned}
\end{equation}
Continuity extends these identities to the closures of the respective faces.

In case~(i), take any $y\in G$ and substitute \eqref{eq:multiplicative-separation} into the bottom equality in \eqref{eq:separated-traces}:
\[
 \alpha_1(y)\beta_1(0)=\alpha_2(y)\beta_2(0).
\]
Because $\beta_j(0)>0$, the constant $c=\beta_1(0)/\beta_2(0)$ is well defined and positive.
The preceding equality gives $\alpha_2=c\alpha_1$ on $G$, and continuity extends this relation to $\overline G$.
We can therefore use it at a point $y_*\in\Sigma\subset\partial G$, which exists because $\Sigma$ is nonempty.
For every $t\in(0,L)$, the point $(y_*,t)$ lies on the observed lateral face, so the second equality in \eqref{eq:separated-traces} gives
\[
 \alpha_1(y_*)\beta_1(t)
 =\alpha_2(y_*)\beta_2(t)
 =c\alpha_1(y_*)\beta_2(t).
\]
Canceling the positive factor $\alpha_1(y_*)$ yields $\beta_2=c^{-1}\beta_1$ on $(0,L)$, and continuity extends this relation to $[0,L]$.
For an arbitrary $(y,t)\in\overline G\times[0,L]$, substituting these two factor relations into \eqref{eq:multiplicative-separation} now gives
\[
 \mu_2(y,t)
 =(c\alpha_1(y))(c^{-1}\beta_1(t))
 =\mu_1(y,t).
\]

In case~(ii), take any $y\in G$ and substitute \eqref{eq:additive-separation} into the bottom equality in \eqref{eq:separated-traces}:
\[
 a_1(y)+b_1(0)=a_2(y)+b_2(0).
\]
With $c=b_1(0)-b_2(0)$, this gives $a_2=a_1+c$ on $G$, and continuity extends the relation to $\overline G$.
Choose $y_*\in\Sigma\subset\partial G$, so this relation also holds at $y_*$.
For each $t\in(0,L)$, apply the lateral equality in \eqref{eq:separated-traces} at $(y_*,t)$ to obtain
\[
 a_1(y_*)+b_1(t)
 =a_2(y_*)+b_2(t)
 =a_1(y_*)+c+b_2(t).
\]
Subtracting $a_1(y_*)$ gives $b_2=b_1-c$ on $(0,L)$, and continuity extends this relation to $[0,L]$.
Thus, for an arbitrary $(y,t)\in\overline G\times[0,L]$, \eqref{eq:additive-separation} gives
\[
 \mu_2(y,t)
 =(a_1(y)+c)+(b_1(t)-c)
 =\mu_1(y,t).
\]
We have therefore proved $\mu_1=\mu_2$ on $\overline\Omega$ in both cases.
The factor and summand relations also prove Remark~\ref{rem:separation-ambiguity}: the normalization $\beta_j(0)=1$ fixes $c=1$ in case~(i), while $b_j(0)=0$ fixes $c=0$ in case~(ii).

Finally, suppose that the full elastic maps agree.
By \eqref{eq:localized-dn-map}, their localized restrictions to $\Gamma$ agree as well, so the argument above gives $\mu_1=\mu_2$ on $\overline\Omega$.
The product geometry, the assumed smoothness, and \eqref{eq:convexity} satisfy the remaining hypotheses of Proposition~\ref{prop:known-shear}, which applies to product domains including their edges and yields $\lambda_1=\lambda_2$ on $\overline\Omega$.
Together with the equality of the shear moduli, this proves $(\lambda_1,\mu_1)=(\lambda_2,\mu_2)$ on $\overline\Omega$.
\end{proof}

\section{Directional quasianalyticity}\label{sec:proof-qa}
By Lemma~\ref{lem:boundary-jets}, equal boundary measurements \eqref{eq:dn} or \eqref{eq:localized-dn-map} give equality of all derivatives of $\mu$ on the observed boundary.
We then apply the identity principle \eqref{eq:qa-definition} along the line sections in Theorem~\ref{thm:qa} and the prescribed curves in Proposition~\ref{prop:curve-qa} to determine $\mu$ throughout $\Omega$.
With full boundary data, Proposition~\ref{prop:known-shear} also determines $\lambda$ from the common shear modulus.

\subsection{Quasianalyticity along lines}\label{subsec:qa-lines}

\begin{proof}[Proof of Theorem~\ref{thm:qa}]
Let $\Omega$ be the bounded connected smooth domain in Theorem~\ref{thm:qa}, and let the two coefficient pairs satisfy its smoothness, strong convexity \eqref{eq:convexity}, and directional quasianalyticity assumptions.
Assume that their full boundary maps defined in \eqref{eq:dn} agree.
The domain and coefficient assumptions allow us to apply Lemma~\ref{lem:boundary-jets} with $\Gamma=\partial\Omega$.
For the difference $h=\mu_1-\mu_2$, defined on the common neighborhood $U$, we obtain
\begin{equation}\label{eq:qa-boundary-vanishing}
 D^\eta h(p)=0,
 \qquad p\in\partial\Omega,\quad \eta\in\mathbb N_0^3.
\end{equation}

To determine $h$ in the interior, fix an arbitrary $x\in\Omega$ and use the direction $e$ prescribed in the theorem.
Write $x=y+s_0e$, where $s_0=x\cdot e$ and $y=x-s_0e\in e^\perp$.
By \eqref{eq:qa-line-section}, $s_0\in I_y$.
Since $\Omega$ is open, $I_y$ is open, so $s_0$ belongs to a unique connected component $(a,b)$ of $I_y$.
Boundedness of $\Omega$ makes $a$ and $b$ finite.
Approaching $a$ from inside $(a,b)$ shows that $y+ae\in\overline\Omega$.
If this point were in $\Omega$, openness would extend the component past $a$, contradicting its maximality.
Thus $y+ae\in\partial\Omega$, where \eqref{eq:qa-boundary-vanishing} applies.

For this component, the quasianalyticity hypothesis supplies a connected open interval $J\supset[a,b]$ with $y+Je\subset U$, on which both functions $s\mapsto\mu_j(y+se)$ belong to the same class $C^{\{M\}}(J)$ from Definition~\ref{def:quasianalytic}.
Define the difference along the line by $h_y:J\to\mathbb R$, $h_y(s)=h(y+se)$.
Linearity of the common class gives $h_y\in C^{\{M\}}(J)$.
At $a$, the chain rule expresses $h_y^{(k)}(a)$ as a derivative of $h$ in the fixed direction $e$ at $y+ae$.
This is a linear combination of Cartesian derivatives of order $k$, all of which vanish by \eqref{eq:qa-boundary-vanishing}; hence
\[
 h_y^{(k)}(a)=(e\cdot\nabla)^kh(y+ae)=0,
 \qquad k\geq0.
\]
Because $J$ is open and contains $[a,b]$, $a$ is an interior point of $J$.
We may therefore apply the identity principle \eqref{eq:qa-definition}, which gives $h_y=0$ on $J$.
In particular, $s_0\in(a,b)\subset J$ implies $h(x)=h_y(s_0)=0$.

Since $x\in\Omega$ was arbitrary, $h=0$ throughout $\Omega$, and continuity gives $\mu_1=\mu_2$ on $\overline\Omega$.
Each point is treated on the component containing it, so the argument also covers line sections with several components.
The common shear modulus, the original full boundary map equality, and the assumed domain and coefficient conditions now satisfy all hypotheses of Proposition~\ref{prop:known-shear}.
That proposition gives $\lambda_1=\lambda_2$ on $\overline\Omega$, completing the equality of both Lam\'e parameters.
\end{proof}

\subsection{Quasianalyticity along curves}\label{subsec:qa-curves}
We now apply the identity principle \eqref{eq:qa-definition} to shear moduli that are quasianalytic along prescribed smooth curves.
When these curves start on the observed boundary and cover $\Omega$, the localized measurements \eqref{eq:localized-dn-map} determine $\mu$ throughout $\Omega$.

\begin{proposition}\label{prop:curve-qa}
Let $\Omega$, $U$, and $(\lambda_j,\mu_j)$, $j=1,2$, satisfy the domain, smoothness, and strong convexity conditions of Theorem~\ref{thm:qa}.
Let $\Gamma\subset\partial\Omega$ be nonempty and relatively open.
Suppose that $\Omega$ is covered by the images $\gamma((a,b))$ of a family of smooth curves $\gamma:J\to U$, where $a<b$ and $J$ is a connected open interval containing $[a,b]$, with
\begin{equation}\label{eq:qa-curve-geometry}
 \gamma(a)\in\Gamma,\qquad
 \gamma((a,b))\subset\Omega,\qquad
 \gamma(b)\in\overline\Omega,\qquad
 \gamma'(s)\ne0\quad(a\leq s\leq b).
\end{equation}
For each curve, assume that both restrictions belong to a common class from Definition~\ref{def:quasianalytic}:
\begin{equation}\label{eq:qa-curve-class}
 \mu_j\circ\gamma\in C^{\{M\}}(J),\qquad j=1,2.
\end{equation}
The curves, their parametrizations, and the weights are prescribed independently of the candidates.
The intervals and weights may vary between curves, and the defining constants need not be uniform.
Then
\[
 \Lambda_{\lambda_1,\mu_1}^{\Gamma}
 =\Lambda_{\lambda_2,\mu_2}^{\Gamma}
 \quad\Longrightarrow\quad
 \mu_1=\mu_2\quad\hbox{on }\overline\Omega.
\]
If the full elastic Dirichlet-to-Neumann maps agree, then both Lam\'e parameters agree on $\overline\Omega$.
No quasianalyticity is required of $\lambda_j$.
\end{proposition}

\begin{proof}
Let the domain, coefficient pairs, observed set $\Gamma$, and prescribed curves satisfy the hypotheses of Proposition~\ref{prop:curve-qa}, and assume that the localized maps \eqref{eq:localized-dn-map} agree.
The domain and coefficient assumptions allow us to apply Lemma~\ref{lem:boundary-jets} on $\Gamma$.
For the difference $h=\mu_1-\mu_2$, defined on $U$, this gives
\begin{equation}\label{eq:qa-local-boundary-vanishing}
 D^\eta h(p)=0,
 \qquad p\in\Gamma,\quad \eta\in\mathbb N_0^3.
\end{equation}

Fix an arbitrary $x\in\Omega$.
By the covering hypothesis, there is a prescribed curve $\gamma:J\to U$ and a parameter $s_0\in(a,b)$ such that $x=\gamma(s_0)$.
For this curve, \eqref{eq:qa-curve-geometry} places the starting point $\gamma(a)$ on the observed boundary $\Gamma$.
Define the difference along the curve by $q:J\to\mathbb R$, $q(s)=h(\gamma(s))$.
By \eqref{eq:qa-curve-class} and linearity of the common class, $q\in C^{\{M\}}(J)$.

For $k\geq1$, repeated differentiation of $q'(s)=\nabla h(\gamma(s))\cdot\gamma'(s)$ expresses $q^{(k)}(a)$ in terms of Cartesian derivatives of $h$ at $\gamma(a)$ of orders $1$ through $k$.
All these derivatives vanish by \eqref{eq:qa-local-boundary-vanishing}, and that identity also gives $q(a)=h(\gamma(a))=0$.
Thus $q^{(k)}(a)=0$ for every $k\geq0$.
The prescribed interval $J$ is open and contains $[a,b]$, so $a$ is an interior point of $J$.
The identity principle \eqref{eq:qa-definition} therefore gives $q=0$ on $J$, and in particular $h(x)=q(s_0)=0$.
Since $x$ was arbitrary, $h=0$ throughout $\Omega$, and continuity yields $\mu_1=\mu_2$ on $\overline\Omega$.

Finally, if the full elastic maps agree, \eqref{eq:localized-dn-map} gives equality of their restrictions to $\Gamma$, so the preceding argument recovers the common shear modulus.
The domain and coefficient assumptions then satisfy all remaining hypotheses of Proposition~\ref{prop:known-shear}, which gives $\lambda_1=\lambda_2$ on $\overline\Omega$.
Thus both Lam\'e parameters agree on $\overline\Omega$.
\end{proof}

\subsection*{Acknowledgments}
The work of L. Chen was partly supported by the National Natural Science Foundation of China (No.~12622108) and a grant from Beijing Institute of Technology (No.~2022CX01002).
The work of Y. Jiang was supported by the Hong Kong RGC Project JRFS2627-1S06.
The work of H. Liu is supported by the Hong Kong RGC General Research Funds (projects 11304224, 11311122 and 11303125).

The authors used AI tools to assist with literature searches and manuscript revision.
The authors independently verified all mathematical arguments and proofs and take full responsibility for the final manuscript.

\raggedbottom
\bibliographystyle{plain}
\bibliography{references}
\end{document}